\title{Radial projections along chains}
\author{Laurent Dufloux}

\documentclass{article}

   \usepackage{amsmath}
    \usepackage{amsthm}
	\usepackage{amssymb}
\newcommand{\R}{\mathbf{R}}
\newcommand{\C}{\mathbf{C}}

\newcommand{\GG}{\mathbf{G}}
\newcommand{\PP}{\mathbf{P}}

\newcommand{\Heis}{\mathcal{H}}
\newcommand{\Hyp}{\mathbf{H}_\mathbf{C}}
\newcommand{\Kor}{Kor\'{a}nyi~}
\newcommand{\pdc}{\mathbf{P}_{\mathbf{C}}}

\newtheorem{remark}{Remark}
\newtheorem{lemma}{Lemma}
\newtheorem{theorem}{Theorem}
\newtheorem{proposition}{Proposition}
\newtheorem{corollary}{Corollary}
\newtheorem{definition}{Definition}

\newtheorem*{theoremA}{Theorem A}
\newtheorem*{theoremB}{Theorem B}

\begin{document}
\maketitle
\abstract{We state strong Marstrand properties for two related families of fractals in Heisenberg groups $\Heis^d$: limit sets 
of Schottky groups in good position, and attractors of self-similar IFS enjoying the open set condition in the quotient $\Heis^d/Z$.
For such a fractal $X$, we show that the dimension of $\pi_x X$ does not depend on $x \in \Heis^d$, where $\pi_x$ denotes the radial 
projection along chains passing through $x$. This follows from a local entropy averages argument due to Hochman and Shmerkin.}

\section{Introduction}
Recall Marstrand's classical projection Theorem in the plane: if $A$ is a Borel subset of $\R^2$, of Hausdorff dimension $s$, then the projection of $A$ in almost every direction 
has Hausdorff dimension $\inf \{1,s\}$; see \cite{Marstrand1954}. This result was proved again by Kaufman \cite{Kaufman1968} using potential methods, then generalized to higher dimensions by Mattila \cite{Mattila1975}, 
\cite{Mattila1995}. Several authors have since been working on 
projection Theorems, strengthening or generalizing this basic result in many ways. One direction of research is to look at special families of fractals, in order to prove deterministic results; 
that is, to compute the dimension of the projection in a fixed direction (or usually an explicit set of directions). 

Let us quote as an example a special case of Hochman-Shmerkin Theorem 1.6 from \cite{Hochman2012}: if $f$ and $g$ are contracting similarities of $\R^2$ whose orthogonal parts generate a dense subgroup of $\mathbf{SO}(2)$
and such that the iterated function system $\{f, g \}$ satisfies the strong separation condition, then, letting $X$ be the attractor of this IFS, for any angle $\theta$
\[ \dim (\pi_\theta X) = \inf\{1, \dim (X)\} \]
where $\pi_\theta$ is the orthogonal projection onto the line of angle $\theta$ in $\R^2$. We may say that $X$ possesses a strong Marstrand property with respect 
to the family of all orthogonal projections (in other words, the exceptional set in Marstrand theorem, 
 is in fact empty in this case).

Here, as in the rest of the paper, we denote by $\dim$ the Hausdorff dimension of a  set.

Theorem 8.1 in \cite{Hochman2012} is the main technical device that allows Hochman and Shmerkin to prove their projection results. Our motivation in this paper is 
to apply the arguments of Hochman and Shmerkin in the setting of Heisenberg groups. Thus Theorem \ref{th.cp-th} below is our attempt at stating a Heisenberg version of Theorem 8.1 
from \cite{Hochman2012}. The reader will soon realize that our Theorem \ref{th.cp-th} is not quite as useful as Hochman-Shmerkin's Theorem 8.1 -- but we nonetheless provide 
some applications below. 

Projection Theorems in Euclidean spaces usually deal with the most obvious mappings: linear projections onto subspaces. In Heisenberg groups, there are not so many subspaces and 
coming up with an interesting family of projections is in itself a non-obvious problem. We refer the reader to \cite{Balogh2012, Balogh2013, Faessler2016, Hovila2014} for some versions 
of Marstrand's projection Theorem. In \cite{Dufloux2017a} we introduced radial projections along chains and showed 
that they satisfy a property known as ``transversality'', which immediately 
gives a version of Marstrand Theorem. 

In \cite{Dufloux2018a}, Ville Suomala and the author looked at random cut-out sets in the first Heisenberg group and computed their dimension 
with respect to both the subriemannian and the Riemannian metric; we also constructed random cut-out set in the boundary of complex hyperbolic plane, $\partial \Hyp^2$, and proved 
that with positive probability such a cut-out set satisfies a strong Marstrand property (w.r.t. radial projections along chains). This is an analogue of the fact that random cut-out sets in Euclidean spaces also satisfy a 
strong Marstrand property (with respect to orthogonal projections) with positive probability.

The main result of this paper, as we said before, is a version of Hochman-Shmerkin projection Theorem, for measures in Heisenberg groups that enjoy a form of ergodic-theoretic self-similarity.
We refer the reader to Theorem \ref{th.cp-th} (for CP-distributions) and Theorem \ref{th.proj-efd} (for Fractal Distributions) for precise
statements. As applications, 
we obtain the following:

\begin{theoremA}
    Let $\Gamma$ be a Schottky subgroup of $\mathbf{PU}(1,d+1)$ ($d \geq 1$) in good position. The limit set $\Lambda_\Gamma$ satisfies 
    the following property:
    for any $\xi \in \partial \Hyp^{d+1}$, the dimension of $\Lambda_\Gamma$ transverse to the foliation by chains passing through $\xi$ is equal 
    to the Poincar\'{e} exponent $\delta_\Gamma$.
\end{theoremA}
Schottky groups in good position were defined in \cite{Dufloux2017}.

\begin{theoremB}
    Let $\mathcal F = \{ f_i\ ;\ 1 \leq i \leq k\}$ be a self-similar iterated function system in $\C^d$ ($d \geq 1$) satisfying the open set 
    condition, and fix a lift 
    $\widetilde{\mathcal F}= \{ \widetilde{f}_i\ ;\ 1 \leq i \leq k\}$
of $\mathcal F$ to the Heisenberg group $\Heis^d$. The attractor $\widetilde{X} \subset \Heis^d$ of $\widetilde{\mathcal F}$
enjoys the following property: for any $x \in \Heis^d$, the dimension of $\widetilde{X}$ transverse to the 
foliation by chains passing through $x$ is equal to 
\[ \dim(X) = \dim(\widetilde{X}) \]
\end{theoremB}

By definition, a Borel set $A$ has transverse dimension $s$ with respect to a foliation $\mathcal F$ if there is a Lipschitz mapping $\pi$
that parametrizes the leaves of $\mathcal F$ and $\dim(\pi A)=s$; see 1.2 in \cite{Dufloux2017a} for a precise definition.

The plan of the paper is as follows: in section \ref{s.heisenberg} we define the classical (complex) Heisenberg groups $\Heis^d$, $d \geq 1$. 
We also define the radial projections along chains; we do this first in the boundary of complex hyperbolic spaces and then 
go back to Heisenberg groups via Heisenberg stereographic projections; we then compute the Pansu derivative of radial projections, this 
is needed in the proof of Theorem \ref{th.cp-th}. We also introduce the Strichartz cubes as a replacement for the familiar dyadic partitions 
in Euclidean spaces, this also is needed in the proof of Theorem \ref{th.cp-th} where we apply Hochman-Shmerkin local entropy averages inequality.
These preliminaries are then applied to the proof of Theorems \ref{th.cp-th} and \ref{th.proj-efd} in section \ref{s.lea}. Finally in section 
\ref{s.applications} we provide two simple examples of fractal distributions in $\Heis^d$ to which we can apply Theorem \ref{th.proj-efd}.

\section{Heisenberg groups}\label{s.heisenberg}
\subsection{Definition}
Fix $d \geq 1$. We endow $\C^d$ with the usual Hermitian inner product 
\[ u \cdot v = \sum_{k=1}^d \overline{u_k} v_k\text{.} \]
We denote by  $\Heis^d$ the (complex) Heisenberg group $\C^d \times \R$
endowed with the group law 
\[ (u,s) \cdot (v,t) = (u+v, s+t-\omega(u,v)) \]
where $\omega : \C^d \times \C^d \to \R$ is the $\R$-linear alternate form defined by 
\[ \omega(u,v) = \mathrm{Im}(u \cdot v) \]
($\mathrm{Im}(z)$ stands for the imaginary part of $z$).

The center $Z = \{0\} \times \R$ of $\Heis^d$ is also equal to its derived subgroup. We denote by $\pi_Z$ the quotient mapping 
\[ \pi_Z : \Heis^d \to \Heis^d/Z \simeq \C^d \]

For $(u,s) \in \Heis^d$, the \Kor gauge is  $\| (u,s) \| = (\| u \|^4 + 4s^2)^{1/4}$; the \Kor metric is defined by 
\[ d(h_1,h_2) = \| h_1^{-1} \cdot h_2 \| \]
for $h_1,h_2 \in \Heis^d$.

The \Kor metric is homogeneous with respect to the Heisenberg dilations
\[ \zeta \cdot (u,s) = (\zeta u, |\zeta|^2 s)\quad \text{for} \ \zeta \in \C \]
meaning $d(\zeta \cdot h_1, \zeta \cdot h_2) = |\zeta| d(h_1,h_2)$ for any $\zeta \in \C$ and $h_1,h_2 \in \Heis^d$.

The Haar measure on $\Heis^d$ will be denoted by $\lambda$. We normalize it in such a way that it coincides with the usual Lebesgue measure on 
$\C^d \times \R$. For any $r > 0$, the push-forward of $\lambda$ through the Heisenberg dilation of ratio $r$ is equal to $r^{2d+2} \lambda$.
In particular, $2d+2$ is the Hausdorff dimension of $\Heis^d$ (with respect to the \Kor metric).

\subsection{Radial projections along chains}
\subsubsection{Chains in $\Heis^d$}
If $f$ is a M\"{o}bius transformation of $\Heis^d \cup \{ \infty \}$, the image $f(Z)$ of the center is called a \emph{chain}. This includes 
Euclidean circles of $\C^d \times \{0\}$ centered at the origin; in fact, any chain is a left translate translate of either $Z$ 
or such a (uniquely defined) circle.

Chains of the  first kind (resp. second) kind are called \emph{infinite} (resp. \emph{finite}) chains. The image, through $\pi_Z$, of a finite 
chain is a Euclidean circle in $\Heis^d / \R \simeq \C^d$.

A basic property of chains is that through any distinct $h_1,h_2 \in \Heis^n$ there passes a unique chain. Any $h_1 \in \Heis^n$ thus yields
a foliation of $\Heis^n \setminus \{h_1 \}$ the leaves of which are the chains passing through $h_1$ (with $h_1$ removed).

\begin{remark}
    The chains we consider are called $\C^1$-chains in \cite{Goldman1999}. We do not consider $\C^k$-chains for $1 < k \leq d$ in this paper.
\end{remark}

\subsubsection{Radial projections}
In order to carry out  computations with chains, it is convenient to work with explicit projections. To every point $h \in \Heis^d$ 
we are going to associate a projection mapping 
\[ \pi_h : \Heis^d \setminus \{ h \} \to \C^d \]
that is Lipschitz on the complement of any compact neighbourhood of $h$, and whose fibers are the chains passing through $h$. 

First this, let us introduce the complex hyperbolic space 
\[ \Hyp^{d+1} = \{ x \in \pdc^{d+1}\ ; \langle x,x \rangle < 0 \} \]
where $\pdc^{d+1}$ is the usual complex projective space of (complex) dimension $d+1$, and $\langle \cdot, \cdot \rangle$ is the Hermitian 
form on $\C^{d+2}$ defined by 
\[ \langle x,y \rangle = \overline{x_0} y_{d+1} + \overline{x_{d+1}} y_0 - \sum_{k=1}^{d} \overline{x_k} y_k \]
The notation $x^\dagger$ denotes the orthogonal, with respect to this Hermitian form, of the complex line spanned by $x \in \C^{d+2}$, $x \neq 0$.

The visual boundary $\partial \Hyp^{d+1}$ is equal to 
\[ \partial \Hyp^{d+1} = \{ x \in \pdc^{d+1}\ ;\ \langle x,x \rangle = 0 \} \]
Let $f_0 = (1,0 \ldots , 0) \in \C^{d+2}$; the ``Heisenberg coordinates'' mapping 
\[ \Phi: \Heis^d \to \partial \Hyp^{d+1} \setminus \{[f_0] \} \]
that maps $(u,s) \in \Heis^d$ to the projective vector 
\[ \left[ \frac{\|u\|^2}{2} + is : \overline{u_1} : \ldots : \overline{u_d} : 1 \right] \in \partial \Hyp^{d+1} \subset \pdc^{d+1} \]
is a locally biLipschitz homeomorphism, when $\partial \Hyp^{d+1}$ is endowed with the usual Gromov-Bourdon metric 
\[ d(x,y) = \sqrt{\frac{| \langle x,y \rangle |}{\|x\| \cdot \|y\|}} \]
For any complex projective line $L \subset \pdc^{d+1}$, if the intersection $L \cap \Hyp^{d+1}$ is non-empty, we say that this intersection 
is a \emph{chain}. The mapping $\Phi$ then defined a bijection between \emph{chains} in $\Heis^d$ and \emph{chains} in $\partial \Hyp^{d+1}$. 

Parametrizing chains in $\Heis^d$ is thus equivalent to parametrizing chains in $\partial \Hyp^{d+1}$. Note that 
infinite chains of $\Heis^d$ correspond to chains passing through $f_0$ in $\partial \Hyp^{d+1}$.

Now fix $x \in \partial \Hyp^{d+1}$ and let us describe a projection mapping $\partial \Hyp^{d+1} \setminus \{x\}$ 
that parametrizes the chains passing though $x$. For a finite-dimensional complex vector space $E$,
we denote by $\GG^k (E)$ the Grassmannian of $k$-vectors of $E$; the exterior (progressive) product will be denoted using the symbol $\vee$, whereas the 
regressive product will be denoted by $\wedge$. The regressive product is well-defined up to the choice of a basis of $E$; here $E=\C^{d+2}$ 
and we choose the usual canonical basis.

Let first $\mathcal Q$ be the $\C$-linear isomorphism $\C^{d+2} \to \GG^{d+1} (\C^{d+2})$ defined by the relation 
\[ \langle x , y \rangle\ \mathbf{f} = \mathcal Q(x) \vee y \]
for all $y \in \C^{d+2}$, where, letting $(f_0,\ldots,f_{d+1})$ be the canonical basis of $\C^{d+2}$,  $\mathbf{f}$ denotes the element 
\[ \mathbf{f} = f_0 \vee f_1 \vee \ldots \vee f_{d+1} \in \GG^{d+2}(\C^{d+2}) \simeq \C \]

Note that $\mathcal Q(x)$ actually belongs to $\GG^{d+1} (x^\dagger)$ by definition. 

Fix distinct non-zero vectors $x,y \in \C^{d+2}$ such that their images in $\pdc^{d+1}$ belong to $\partial \Hyp^{d+1}$;
the regressive product $\mathcal Q(x) \wedge \mathcal Q(y)$ is an element of $\GG^{d} (x^\dagger)$
(as well as of $\GG^{d}(y^\dagger)$ but we choose to consider $x$ fixed and we see $y$ as a variable) that belongs to 
 the complement of the vector subspace 
\[ \mathcal A(x) = \{ \mathbf{u} \in \GG^d (x^\dagger)\ ;\ x \vee \mathbf{u} = 0 \} \]

In other words, for $x \in \partial \Hyp^{d+1}$ fixed,  $y \mapsto [ \mathcal Q(x) \wedge \mathcal Q(y) ]$ defines a mapping 
from $\partial \Hyp^{d+1} \setminus \{x\}$ to $\PP_\C (\GG^d (x^\dagger)) \setminus \PP (\mathcal A(x))$. This mapping parametrizes the chains 
passing through $x$.

We are almost done: since $\PP_\C (\GG^d (x^\dagger))$ is a $d$-dimensional projective line, and $\mathcal A(x)$ is a hyperplane of
 $\GG^d (x^\dagger)$, the complement $\PP_\C (\GG^d (x^\dagger)) \setminus \PP_\C(A(x))$ identifies
  with the affine space $\mathcal A(x) \simeq \C^d$. An explicit identification is obtained as follow: let $\hat{\mathbf{x}}$ 
be an non-zero element of $\GG^d (x^\dagger)$ orthogonal to $\mathcal A(x)$ with respect to the Hermitian inner product on $\GG^d (x^\dagger)$.
The mapping 
\[ [\mathbf{u}] \mapsto \frac{\mathbf{u}}{\hat{\mathbf{x}} \cdot \mathbf{u}} - \hat{\mathbf{x}} \]
is an affine isomorphism from $\PP_\C (\GG^d (x^\dagger)) \setminus \PP_\C(A(x))$ to $\mathcal A(x)$. Note that $\mathbf{u}$ is an element of 
$\GG^d (x^\dagger))$ and the right-hand side depends only on the projective class $[\mathbf{u}]$.

At this point we have, for a fixed $x \in \partial \Hyp^{d+1}$, a mapping $\partial \Hyp^{d+1} \to \mathcal A(x) \simeq \C^d$ that parametrizes the chains
passing through $x$. Now  compose this mapping with the Heisenberg coordinates $\Phi: \Heis^d \to \partial \Hyp^{d+1}$ to obtain 
the required projection mapping parametrizing chains passing through $\Phi^{-1}(x)$ in $\Heis^d$.

Explicitly, for $h \in \Heis^d$, and $x = \Phi(h)$ as above,  we have 
\[ \pi_h (h') = \frac{\mathcal Q(x) \vee \mathcal Q (\Phi(h'))}{\hat{\mathbf{x}} \cdot (\mathcal Q(x) \vee \mathcal Q (\Phi(h')))} -
 \hat{\mathbf{x}} \in \mathcal A(\Phi(h)) \]

This mapping is defined in $\Heis^d \setminus \{ h \}$. 

\begin{lemma}
    For any $h \in \Heis^d$, the mapping $\pi_h$ defined above satisfies the following:
    \begin{enumerate}
        \item $\pi_h$ parametrizes the chains passing through $h$: for any $u \in \mathcal A(\Phi(h))$, $\pi_h^{-1}(u)$ 
        is a chain passing through $h$ (with $h$ removed);
        \item $\pi_h$ is locally Lipschitz.
        \item If $T : \Heis^d \to \Heis^d$ is a Heisenberg similarity transformation, there is a locally bilipschitz homeomorphism 
        $f : \mathcal A(\Phi(T^{-1}(h))) \to \mathcal A(\Phi(h))$ such that $\pi_h \circ T = f \circ \pi_{T^{-1}(g)}$.
    \end{enumerate}
\end{lemma}
Note that infinite chains are parametrized by $\pi_Z$, which corresponds to the case when $h$ is the point at infinity, \emph{i.e.} 
in $\partial \Hyp^{d+1}$ we are looking at chains passing through the point $[f_0]$.

\subsubsection{Pansu derivative of the radial projections}
Recall the definition of Pansu derivative: if $U$ is an open subset of $\Heis^d$, a mapping $f : U \to \R^k$ is differentiable 
at $h_0 \in U$ if the maps 
\[ h \mapsto \frac{f(h_0 \cdot (r \cdot h)) - f(h_0)}{r} \]
converge, as $r \to 0$, uniformly on compact subsets of $\Heis^d$, to a group homomorphism $\mathrm{D} f(h_0)$ satisfying the homogeneity condition 
\[ \mathrm{D} f(h_0) (r \cdot h) = r \mathrm{D} f(h_0) (h) \]
for all $r > 0$.

Any Lipschitz mapping $\Heis^d \to \R^k$ must be Pansu-differentiable almost everywhere. In particular this 
 holds for the projection $\pi_{h_0}$ defined in the previous paragraph.  
 In fact, a straightforward (if tedious) computation 
shows that $\pi_{h_0}$ is differentiable everywhere in $\Heis^d \setminus \{ h_0 \}$ and that the derivative $\mathrm{D} \pi_{h_0} (h)$ 
is a continuous function of $h$:

\begin{lemma}\label{lemma.pansu-derivative}
    Fix $h_0 \in \Heis^d$. For any $h \in \Heis^d \setminus \{ h_0 \}$ there is a $\C$-linear isomorphism 
    \[ M(h_0,h) : \Heis^d/Z \simeq \C^d \to \mathcal{A}(\Phi(h_0)) \]
    such that the Pansu derivative of $\pi_{h_0}$ at $h$ is given by 
    \[ \mathrm{D} \pi_{h_0} (h) = M(h_0,h) \circ \pi_Z \]
    Furthermore, 
    \begin{itemize}
        \item         for $h'$ in a neighbourhood of $h$
    \[ \pi_{h_0} (h') = \pi_{h_0}(h) + \mathrm{D} \pi_{h_0}(h) (h^{-1} h') + O(d(h,h')^2) \]
    \item 
    $M(h_0,h)$ depends continuously on $h$.
    \end{itemize}
\end{lemma}

Note that the fact that $\mathrm{D} \pi_{h_0}$ is essentially the quotient mapping $\pi_Z$ is not surprising: any group homomorphism from 
$\Heis^d$ onto $\C^d$ must be of the form $A \circ \pi_Z$ for some isomorphism $A$. The point is that $A$ here is continuous in $h$, which 
is very intuitive.

\subsection{Stricharz cubes}\label{ss.cubes}
One of the main technical tools used in \cite{Hochman2012} (and which we use again in this paper) is a \emph{local entropy averages} formula 
that relies crucially on the existence of good partitions of Euclidean spaces (\emph{i.e.} dyadic cubes and variants thereof).

The dyadic partitions coming from the identification of $\Heis^d$ with $\C^d\times \R$ are of no use in our situation because these partitions 
are not translation invariant:  the image $h \cdot Q$ of a Euclidean cube through a Heisenberg translation is not a cube, because $h$ 
``tilts'' $Q$ in the vertical direction. 

Fortunately, a good substitute for dyadic partitions in $\Heis^d$ has been introduced by Strichartz. An odd integer $b \geq 2d+1$ 
being fixed, there exists a compact subset $T$ of $\Heis^d$, the ``unit Strichartz cube'' (or ``tile'') satisfying the following proposition.

\begin{proposition}[\cite{Tyson2008}]
    \begin{enumerate}
        \item The origin $0$ of $\Heis^d$ is an interior point of $T$.
    \item The closure of the interior of $T$ is equal to $T$.
    \item $T$ is a fundamental domain for the operation of 
    \[ \Gamma = \{ (u,s) \in \Heis^d\ ;\ u \in \mathbf{Z}^d \oplus i \mathbf{Z}^d,\ s \in \mathbf{Z} \} \]
     on $\Heis^d$:
    \[ \Heis^d = \Gamma \cdot T = \bigcup_{\gamma \in \Gamma} \gamma \cdot T \]
    and for $\gamma,\gamma' \in \Gamma$ distinct, $\gamma T$ and $\gamma' T$ have disjoint interiors.
    \item The boundary $\partial T$ is Lebesgue-negligible.
    \item The blow-up $T_{-1} = b^{-1} \cdot T$ is a (finite) union of translates of $T$.
\end{enumerate}
\end{proposition}
The last property, which comes from the fact that $T$ is in fact the limit set of a self-similar IFS in $\Heis^d$, is of course 
crucial as we want to have a nested family of partitions that is both scale-invariant and translation-invariant.

We let $\mathcal Q_0$ be the partition of $\Heis^d$ with atoms all the translates $\gamma \cdot T$, $\gamma \in \Gamma$. For any $m \geq 2$, 
let also $\mathcal Q_m$ be the partition of $\Heis^d$ that is the image of $\mathcal Q_0$ through the Heisenberg dilation of ratio $b^{-m}$, 
so that $\mathcal Q_{m+1}$ refines $\mathcal Q_m$ for every $m \geq 1$: every atom of $\mathcal Q_{m+1}$ is contained in a unique atom of 
$\mathcal Q_m$. Also, any atom $Q \in \mathcal Q_m$ is comparable to a ball of radius $b^{-m}$; more precisely, 
there is a unique element $\gamma \in b^{-m} \cdot \Gamma$ such that 
\[ B(\gamma \cdot 0, C^{-1} b^{-m}) \subset Q \subset B(\gamma \cdot 0, C b^{-m}) \]
where $C>1$ is some uniform constant.

As usual, the $\mathcal Q_m$-atom containing a point $x \in \Heis^d$ will be denoted by $\mathcal Q_m(x)$.

\section{Local entropy averages}\label{s.lea}
\subsection{Definition}

We define CP-distribution in Heisenberg groups following Hochman \cite{Hochman2010}, replacing only Euclidean dyadic cubes with 
Strichartz cubes, and Euclidean dilations with Heisenberg dilations.

As in \ref{ss.cubes}, fix an odd integer $b \geq 2d+1$ and consider the corresponding Strichartz cube $T$ as well as family of nested partitions
 $\mathcal Q_m$, $m \geq 1$. If $\mu$ is a Radon measure such that $\mu(T)>0$, we let 
 \[ \mu^* = \frac{\mu}{\mu(T)}\quad ; \quad \mu^\square = \mu_T \]
 where as usual $\mu_T$ is the probability measure proportional to the restriction $\mu|T$. The corresponding spaces of Radon measures 
 are denoted by $\mathcal M^*$ and $\mathcal M^\square$.

 For any $Q \in \mathcal{Q}_m$, we let $T_Q$ be the unique affine Heisenberg dilation mapping $\overline{Q}$ onto $T$; 
 this is the composition of a Heisenberg dilation of ratio $b^m$ and a Heisenberg translation. 
 For any $\mu \in \mathcal M^\square$ and any $Q \in \mathcal Q_m$ 
 such that $\mu(Q) > 0$, we denote by $\mu_Q$ the conditional measure 
 \[ \mu_Q = \frac{\mu|Q}{\mu(Q)} \]
 and by $\mu^Q$ the push-forward of $\mu_Q$ through $T_Q$. By definition, $\mu^Q$ belongs to $\mathcal M^\square$. 

For any $(\mu,x) \in \mathcal M^\square \times T$ such that $\mu(\mathcal Q_1(x))>0$, let 
\[ M(\mu,x) = (\mu^{\mathcal Q_1(x)}, T_{\mathcal Q_1(x)}x) \in \mathcal M^\square \times T \]

\begin{definition}[\cite{Hochman2012}, \cite{Hochman2010}]
    A Borel probability measure $P$ on $\mathcal M^\square \times T$ is a CP-distribution if 
    \begin{enumerate}
        \item $P$ gives full measure to the Borel set of $(\mu,x)$ such that $\mu(\mathcal Q_1(x))>0$;
        \item $P$ is $M$-invariant;
        \item $P$ possesses the following adaptedness property:
        \[ P =\int \mathrm{d}P (\mu,x)\ \mathrm{Dirac}(\mu) \otimes \mu \]
    \end{enumerate}
\end{definition}
Note that whenever we talk of a CP-distribution, it is implied that an odd integer $b \geq 2d+1$ has been fixed once and for all. 

We say that a CP-distribution $P$ is ergodic if it is ergodic with respect to $M$.

An extended CP-distribution is a Borel probability measure $P$ on $\mathcal M^* \times T$ whose push-forward through the mapping 
$(\mu,x) \mapsto (\mu^\square,x)$ is a CP-distribution.

\subsection{Dimension of projections}

\begin{lemma}\label{lemma.cp-th}
    Let $P$ be an ergodic CP-distribution on $\Heis^d$.    
    For $P$-almost every $\mu$, the following holds:
    for every $h \in \Heis^d$ outside of the unit Strichartz cube $T$,
    \[ \dim(\pi_h \mu) \geq \int \mathrm{d} P(\nu) \dim(\pi_Z \nu) \]
\end{lemma}

\begin{proof}
     We are going to show that for  $P$-almost every $(\mu,x)$
    there is a Strichartz cube $Q \in \mathcal Q_m$, $m \geq 1$, containing $x$ and of positive $\mu$-measure, such that 
    \[ \dim (\pi_h \mu_Q) \geq  \int \mathrm{d} P(\nu) \dim(\pi_Z \nu) 
    \]
    for all $h \notin T$. From this the conclusion follows.
    
    We denote by $H_\rho (\nu)$ the $\rho$-scale entropy of a Borel probability measure $\nu$:
    \[ H_\rho (\nu) = \int \mathrm{d} \nu(x) \log \nu(B(x,\rho)) \]

    Let us first argue as in \cite{Hochman2012}; fix $\nu$, a probability measure supported on the unit Strichartz cube $T$, and fix 
    a neighbourhood $U$ of $T$. I claim that for all $h \notin U$, 
    \begin{equation} \label{eq.proof} \dim (\pi_h \nu) \geq \frac{1}{q \log b} \underset{x \sim \nu}{\mathrm{essinf}} \liminf_{N \to \infty} 
        \sum_{n=0}^{N-1} H_{b^{-q}}
    (\pi_Z \nu^{\mathcal Q_{nq}(x)}) - O(q^{-1}) \end{equation}
    for all $q \geq 1$, where the big-O constant is uniform in $h$. Indeed, as in the proof of Theorem 8.1 in \cite{Hochman2012}, we have first 
    \[ \dim (\pi_h \nu) \geq \frac{1}{q \log b} \underset{x \sim \nu}{\mathrm{essinf}} \liminf_{N \to \infty} 
\sum_{n=0}^{N-1} H_{b^{-q(n+1)}} (\pi_h \nu_{\mathcal{Q}_{nq}(x)}) - O(q^{-1}) \]
for any $q \geq 1$, for all $h$ in the complement of $U$. Now, to show that $H_{b^{-q(n+1)}} (\pi_h \nu_{\mathcal{Q}_{nq}(x)})$ is comparable 
to $H_{b^{-q}}(\pi_Z \nu^{\mathcal Q_{nq}(x)})$ (pay attention to the different mappings $\pi_h$ and $\pi_Z$), recall that $\pi_h$ satisfies 
\[ \pi_h(y) = \pi_h(x) + \mathrm{D} \pi_h (x) (x^{-1} y) + O(d(x,y)^2) \]
so that 
\[ H_{b^{-q(n+1)}} (\pi_h \nu_{\mathcal Q_{nq}(x)}) = H_{b^{-q(n+1)}} \left(\mathrm{D} \pi_h(x) \circ L_{x^{-1}} (\nu_{\mathcal Q_{nq}(x)}) \right) 
+O(1) \]
where $L_{x^{-1}}$ is the left multiplication by $x^{-1}$ in $\Heis^d$. 

The Pansu derivative $\mathrm{D} \pi(x)_h$ is a morphism of Carnot groups; hence, 
\[ H_{b^{-q(n+1)}} \left(\mathrm{D} \pi_h(x) \circ L_{x^{-1}} (\nu_{\mathcal Q_{nq}(x)}) \right) =
 H_{b^{-q}} \left( \mathrm{D} \pi_h (x)  \left(\nu^{\mathcal Q_{nq} (x)}\right) \right) \]

Finally, by Lemma \ref{lemma.pansu-derivative}
\[ H_{b^{-q}} (\mathrm{D} \pi_h(x) (\nu^{\mathcal Q_{nq} (x)})) = H_{b^{-q}} (\pi_Z \nu^{\mathcal Q_{nq} (x)}) + O(1) \]
as $x$ stays in the (compact) Strichartz cube $T$. Thus we have obtained \eqref{eq.proof}.

Now we are ready to work with $P$. For any $q \geq 1$, let $e_q$ be the Borel mapping $\mathcal M^\square \times T \to \R$
 \[ e_q(\nu,x) = H_{b^{-q}} (\pi_Z \nu) \]
and let $\mathbf{X}_q$ be the set of all $(\nu,x) \in \mathcal M^\square \times T$ such that for some integer $i$, $0 \leq i \leq q-1$,
\[ \liminf_{N \to \infty} \frac{1}{N} \sum_{n=0}^{N-1} e_q(M^{nq+i}(\nu,x)) \geq \int e_q\ \mathrm{d}P  \]
It follows from the ergodicity of $P$ (w.r.t $M$), and the ergodic decomposition theorem applied to $P$ w.r.t. $M^q$, that $P(\mathbf{X}_q)=1$ for all $q$, and so the intersection $\mathbf{X} = \cap_q \mathbf{X}_q$
also has full measure. Also, adaptedness of $P$ implies that for $P$-almost every $\mu$, $(\mu,x)$ belongs to $\mathbf{X}$ for $\mu$-almost every $x$.

Pick such a $\mu$. By definition of the transformation $M$, for all $q \geq 1$ there is $0 \leq i \leq q-1$ such that 
\[ \liminf_{N \to \infty} \frac{1}{N} \sum_{n=0}^{N-1} H_{b^{-q}} ((\mu^Q)^{\mathcal Q_{nq}(x)}) \geq \int e_q\ \mathrm{d}P\]
for $\mu^Q$-almost every $x$, where we let $Q=\mathcal Q_i(x)$. By \eqref{eq.proof} this implies 
\[ \dim (\pi_h \mu^Q) \geq \int e_q\ \mathrm{d}P - O(q^{-1}) \]
for all $h \notin U$. Also,  $\dim (\pi_h \mu^Q) = \dim (\pi_{h'} \mu_Q)$ where $h'=T_Q^{-1}(h)$ by Lemma \ref{lemma.pansu-derivative};
 thus we have obtained 
\[ \dim (\pi_h \mu_Q) \geq \int e_q \ \mathrm{d} P - O(q^{-1}) \]
for all $h \notin T_Q^{-1}(U)$; in particular also for all $h \notin U$ (note that $T_Q U \subset U$).

To conclude, observe that for all $q \geq 1$,  $\mu$-almost every $x$ belongs to a Strichartz cube $Q$ such that the above inequality holds
for all $h$. We deduce that the same inequality holds for $\dim (\pi_h \mu)$; the conclusion of the Theorem then follows, for $h \notin U$,
from letting $q \to \infty$, applying Fatou lemma, and the standard fact that lower entropy dimension is always greater than 
Hausdorff dimension, see \cite{Fan2002}; by shrinking the neighbourhood $U$ of $T$ we have the result for all $h \notin T$.

\end{proof}

\begin{theorem}\label{th.cp-th}
    Let $P$ be an ergodic CP-distribution such that $P$-almost every $\mu$ is atomless. Then for $P$-almost every $\mu$, the following holds:
    for every $h \in \Heis^d$,
\begin{equation}\label{eq.hs} \dim(\pi_h \mu) \geq \int \mathrm{d} P(\nu) \dim(\pi_Z \nu) \end{equation}
\end{theorem}
\begin{proof}
    Let $\alpha$ be the right-hand side. By virtue of the Lemma (and using the $M$-invariance of $P$), we know that  a $P$-typical $\mu$ satisfies the following:
    for every $m \geq 1$, and every $Q \in \mathcal Q_m$ such that $\mu(Q)>0$,
    \[ \dim (\pi_h \mu^Q) \geq \alpha \]

    Now fix a $P$-typical $\mu$ and let $h$ belong to $T$. For any $m \geq 1$, we can write 
    \[ \mu = \mu_{\mathcal Q_m(h)} + \sum_{Q} \mu_Q \]
    where in the sum $Q$ goes through all atoms of $\mathcal Q_m$ that have positive $\mu$-measure and do not contain $h$. For 
    such a $Q$, we have 
    \[ \dim (\pi_h \mu_Q) = \dim (\pi_{h'} \mu^Q) \geq \alpha \]
    where $h' = T_Q^{-1}h$ does not belong to $T$. It follows that the push-forward of $\sum_Q \mu_Q$ through $\pi_h$ also has dimension $\geq \alpha$.

    It follows easily that $\dim (\pi_h \mu)$ must be $\geq \alpha$, using the fact that $\mu(\{h\})=0$ by our assumption.
\end{proof}

It is perhaps useful to highlight some of the differences between Theorem \ref{th.cp-th} here and Theorem 8.1 in \cite{Hochman2012}. First note 
that the right-hand side in \eqref{eq.hs} does not depend on $h$. The point is that at very small scale, all the projections $\pi_h$ look 
like $\pi_Z$ (and uniformly so). This is very different from the Euclidean situation, where we look at linear projections which already behave well 
with respect to Euclidean dilations. For this reason, it does not seem to be possible to obtain an analogue of Theorem 8.2 from \cite{Hochman2012}.

A more accurate analogue of Theorem \ref{th.cp-th} in the Euclidean case would deal with projections along spheres at a given angle. 
More precisely, for any point $x \in \R^d$ and any angle $\theta \in \mathbf{P}(\R^d)$, there is a foliation of $\R^d \setminus \{x\}$
whose leaves are the Euclidean circles passing through $x$ and tangent to $\theta$ at $x$. If we look at the derivative of the 
corresponding mapping $\R^d \setminus \{ x\} \to \R^{d-1}$, we get something which resembles the mapping onto the space orthogonal to $\theta$.

Applying the argument of Hochman-Shmerkin to this situation, we would obtain the inequality 
\[ \dim (\pi_{x,\theta} (\mu)) \geq \int \mathrm{d}P(\nu) \dim(\pi_\theta \nu) \]

As mentioned in the introduction to this paper, Theorem \ref{th.cp-th} falls somewhat short of being as useful as Theorem 8.1 from 
\cite{Hochman2012}. We cannot  obtain semi-continuity for the dimension of projections, which is key in applying 
Marstrand Theorem (of which a version for the radial projections we consider was proved in \cite{Dufloux2017a}) to obtain a 
big open set of projections where the dimension is large.

This is why the only applications we can provide (in section \ref{s.applications}) deal with situations where we already 
understand the dimension of $\pi_Z \mu$, and Theorem \ref{th.cp-th} yields a corresponding bound on $\pi_h \mu$ for every $h$.

\subsection{Fractal distributions and centering of CP-distributions}
We define Fractal Distributions in Heisenberg groups following the definition given in \cite{Hochman2010} for Euclidean spaces.
For any $x \in \Heis^d$, let $T_x$ be the (left) Heisenberg translation that maps $x$ to the origin; for any $t \in \R$, let 
$S_t$ be the Heisenberg homothety of ratio $e^t$, 
\[ x =(u,s) \mapsto e^t \cdot x = (e^t u, e^{2t} s)\]

If $\mu$ is a Borel measure whose support contains $x$, we let $T_x^* \mu$ be the measure $(T_x \mu)^*$.

\begin{definition}[\cite{Hochman2010}]
    A fractal distribution is a Borel probability measure $P$ on $\mathcal M^*$ that satisfies the following conditions.

    \begin{enumerate}
\item  Given any relatively compact neighbourhood $U$ of the origin in $\Heis^d$, the distribution 
\[ \int \mathrm{d}P(\mu) \int_U \mathrm{d}\mu(x)\ \mathrm{Dirac}(T_x^* \mu) \]
is equivalent to $P$ ($P$ is ``quasi-Palm'').
\item For any $t \in \R$, $S_t^* P=P$.
    \end{enumerate}
\end{definition}

Note that if $P$ is ergodic, $P$-almost every $\mu$ is exact dimensional of dimension 
\[ \dim P = \int \mathrm{d}P (\nu) \dim(\nu) = \int \mathrm{d} P (\nu) \frac{\log B(0,r)}{\log r} \]
for any $0 < r < 1$.

\begin{theorem}\label{th.proj-efd}
    Let $P$ be a Fractal Distribution in $\Heis^d$. We assume that $P$ is $S_{\log b}^*$-ergodic for some $b \geq 2d+1$ and that $\dim (P) > 0$.
     Then for $P^\square$-almost every $\mu$, 
    \[ \dim (\pi_h \mu) \geq \int \mathrm{d} P^\square(\nu) \dim (\pi_Z \nu) \]
    for all $h \in \Heis^d$.
\end{theorem}
Here $P^\square$ is the (well-defined) push-forward of $P$ through the mapping $\mu \mapsto \mu^\square$.
\begin{proof}
The point is that there exists an extended CP-distribution $Q$ such that $P$ is the push-forward of $Q$ through the ``discrete centering'' map 
\[ (\mu, x) \mapsto T^*_x \mu \]
See \cite{Hochman2010} Theorem 1.15; the proof in the Euclidean cases adapts to the Heisenberg setting in a straightforward way.

Let $Q = \int \mathrm{d} \mathbb{P} (\omega) Q_\omega$ be the ergodic decomposition of $Q$ (with respect to $M$).
Clearly
\[ \int \mathrm{d} \mathbb{P}(\omega) \int \mathrm{d} Q_\omega(\nu) \dim(\pi_Z \nu^\square) = \int \mathrm{d} Q(\nu) \dim(\pi_Z \nu^\square) \]
and the $S_{\log b}^*$-ergodicity of $P$ implies, by an easy argument, that the right-hand side is also equal to 
\[  \int \mathrm{d} P(\nu) \dim(\pi_Z \nu^\square) \]
(we also use the fact that $\pi_Z$ is a group homomorphism). Let us denote by $\alpha$ the common value of these integrals.

For $Q$-almost every $(\mu,x)$, if $Q_\omega$ is the ergodic 
component of $Q$ generated by $(\mu,x)$, 
\[ \dim (\pi_h \mu^\square) \geq \int \mathrm{d} Q_\omega(\nu) \dim(\pi_Z \nu^\square) \]
for all $h \in \Heis^d$ simultaneously by Theorem \ref{th.cp-th}; hence for any fixed $\varepsilon > 0$, if we pick $\mu$ at random according to $Q$, there is positive probability that 
\begin{equation} \label{eq.mu-ergd}
     \dim (\pi_h \mu^\square) \geq \alpha - \varepsilon 
\end{equation}
for all $h$. Because $P$ is the discrete centering of $Q$, we see that the same assertion holds 
if we pick $\mu$ at random according to $P$. 

The set of all $\mu$ such that \eqref{eq.mu-ergd} holds for all $h$ is $S_{\log b}$-invariant; by ergodicity of $P$, it must have full measure. We 
conclude by letting $\varepsilon \to 0$ along a countable sequence.
\end{proof}

\begin{remark}
    In general the inequality in Proposition \ref{th.proj-efd} can very much be strict. Indeed let $P$ be the Dirac 
    mass $P=\mathrm{Dirac}(\mu)$ where $\mu$ is the (suitably normalized) Lebesgue measure on $Z$. Then $P$ is a fractal distribution 
    ergodic with respect to $S_t^*$ for any $t \neq 0$. The projection $\pi_Z \mu$ is a Dirac mass, so it has dimension $0$, whereas 
    for any $h \in \Heis^d$ outside $Z$ the projection $\pi_h \mu$ is absolutely continuous with respect to the Lebesgue measure on a 
    smooth curve, so it has dimension $1$.
\end{remark}

\section{Examples}\label{s.applications}

\subsection{Limit sets of Schottky groups in good position}
\subsubsection{Patterson-Sullivan measures}
Let $G$ be the group $\mathbf{PU}(1,d+1)$ of isometries of the complex hyperbolic space $\Hyp^{d+1}$. Fix a discrete subgroup $\Gamma$ in $G$ and assume that $\Gamma$ is Zariski-dense 
and has finite Bowen-Margulis-Sullivan measure (for example, it may be convex-cocompact). Denote by $\delta_\Gamma$ the Poincar\'{e} exponent of $\Gamma$, $0 < \delta_\Gamma \leq 2d+2$.
The limit set $\Lambda_\Gamma$ is the set of accumulation points of any orbit $\Gamma \cdot O$, $O \in \Hyp^{d+1}$; it is a closed subset of the boundary: 
$\Lambda_\Gamma \subset \partial \Hyp^{d+1}$. The Hausdorff dimension of $\Lambda_\Gamma$, with respect to the visual Gromov-Bourdon metric on the boundary, is equal to $\delta_\Gamma$.

The limit set is the support of the classical family of Patterson-Sullivan measures $(\mu_x)_{x \in \Hyp^{d+1}}$. The measures are pairwise equivalent, and for $x,y \in \Hyp^{d+1}$ the 
Radon-Nikodym derivative is given by 
\[ \frac{\mathrm{d} \mu_y}{\mathrm{d} \mu_x}(\xi)= e^{-\delta_\Gamma b_\xi (y,x)} \]
where $b_\xi(y,x)$ is the usual Busemann function. These measures also have Hausdorff dimension $\delta_\Gamma$.

Let $m$ be the Bowen-Margulis-Sullivan on the frame bundle $\Gamma \backslash G$. We fix an Iwasawa decomposition $G=KAN$ as well as 
an identification $N \simeq \Heis^d$
and we disintegrate $m$ along $N$; this yields a Borel mapping $\sigma$ from $\Gamma \backslash G$ to the space of projective Radon measures on $\Heis^d$; see \cite{Dufloux2017} for details.

As in \cite{Dufloux2018}, we introduce the distribution on $\mathcal M^*$
\[ P = \int \mathrm{d}m (x) \mathrm{Dirac}(\sigma(x)^*) \]
and this is again a fractal distribution, ergodic with respect to $S_t$ for any $t \neq 0$. In \cite{Dufloux2018} we deal with the setting of the real hyperbolic space; the proof 
is identical in the complex hyperbolic case.

Let us make the distribution $P$ more explicit. To pick $\nu$ according to $P$, first choose $\xi,\eta \in \partial \Hyp^{d+1}$ according to the Patterson-Sullivan measure $\mu$ (which is 
atomless); then identify $\partial \Hyp^{d+1} \setminus \{\eta \}$ with $\Heis^d$ by sending $\xi$ to the origin and $\eta$ to infinity (via the Heisenberg stereographic projection). The 
measure $\nu$ is then equal to the push-forward, through this identification $\partial \Hyp^{d+1} \setminus \{\eta \} \simeq \Heis^d$, of the Radon measure $f \mu$, where $f$ is some continuous 
density we do not need to care about here.

\subsubsection{Radial projections of limit sets}
In \cite{Dufloux2017} we showed a Ledrappier-Young formula for conditional measures along group operations; this applies here to the distribution $P$ and yields the following result: 
for any relatively compact neighbourhood $U$ of the origin in $\Heis^d$, and $P$-almost every $\nu$, the Hausdorff dimension of the push-forward  measure $\pi_Z \nu_U$ 
is almost surely equal to 
\[ \int \mathrm{d} P^\square (\theta) \dim (\pi_Z \theta) \]
This number is called the \emph{transverse dimension of $\mu$ along $Z$}. It is also equal to $\dim (\pi_x\mu)$ for $\mu$-almost every 
$x \in \partial \Hyp^{d+1}$, where $\pi_x : \partial \Hyp^{d+1} \setminus \{x\} \to \mathcal A(x)$ now denotes the radial projection along
 chains passing through $x$ in $\partial \Hyp^{d+1}$.

\begin{theorem}
    Let $\Gamma$ be a discrete Zariski-dense subgroup of $G$, with finite BMS measure. Let $\mu$ be the Patterson-Sullivan measure of 
    exponent $\delta_\Gamma$. The radial projection of $\mu$ at \emph{any} point of $\partial \Hyp^{n+1}$ has dimension at least equal 
    to the transverse dimension of $\mu$ along $Z$. 
\end{theorem}
\begin{proof}
    We know by Theorem \ref{th.proj-efd} that if $\nu$ is a $P$-typical measure, then 
\[ \dim (\pi_h \nu) \geq \int \mathrm{d} P^\square (\theta) \dim (\pi_Z \theta) \]
for all $h \in \Heis^d$. Also, $\nu$ is equivalent to the push-forward of $\mu$ through some Heisenberg stereographic projection $\partial \Hyp^{d+1} \setminus \{ x \} \to \Heis^d$;
hence for all $y \in \partial \Hyp^{d+1} \setminus \{ x \}$, it holds that 
\[ \dim (\pi_y \mu) \geq \int \mathrm{d} P^\square (\theta) \dim (\pi_Z \theta) \]
\end{proof}

The problem of computing the transverse dimension of the Patterson-Sullivan measure along $Z$ was raised in \cite{Dufloux2017} and 
was our main motivation for looking to generalize Hochman-Shmerkin methods to Heisenberg settings. It does not seem that 
Theorem \ref{th.proj-efd} sheds any new light on this matter. On the other hand, there is a class of discrete groups 
for which computing the transverse dimension along $Z$ is not difficult; for this class of group (which we call ``Schottky groups in good position''), we obtain for free 
a strong Marstrand property:

\begin{corollary}
If $\Gamma$ is a Schottky subgroup of $G$ in good position, the radial projection of $\Lambda_\Gamma$ at any point of $\Hyp^{n+1}$ has dimension $\delta_\Gamma$.
The same holds for the Patterson-Sullivan measure.
\end{corollary}
\begin{proof}
    We have $\int \mathrm{d} P^\square (\theta) \dim (\pi_Z \theta)$, see \cite{Dufloux2017}.
\end{proof}

\subsection{Self-similar sets}
Let $\mathcal F = \{ f_1,\ldots,f_k \}$ be a family of contractive Heisenberg similarity transformations: 
there are real numbers $r_1,\ldots,r_k \in\ ]0,1[$ such 
that 
\[ d(f_i(x),f_i(y))=r_i d(x,y) \]
for $1 \leq i \leq k$ and $x,y \in \Heis^d$. Explicitly, each $f_i$ is a composition 
\[ f_i = \tau_i \circ h_i \circ u_i \]
where $u_i$ is a \emph{Heisenberg rotation about the vertical axis}, meaning
\[ u_i (u,s) = (U_i\cdot u,s) \]
where $U_i \in \mathbf{U}(d)$; $h_i$ is a Heisenberg homothety $x \mapsto r_i \cdot x$, and $\tau_i$ is the left translation by some element $x_i$ of $\Heis^d$.

We denote by $H$ the closed subgroup of $\mathbf{U}(d)$ generated by the $U_i$, $1 \leq i \leq k$. 

It is well-known that there is a unique compact subset $X$ of $\Heis^d$ that is invariant by $\mathcal F$: $f_i(X) \subset X$ for each $i$.

Let 
$\Lambda = \{ 1,\ldots,k \}$ and $\phi : \Lambda^{\mathbf{N}} \to \Heis^d$ be the coding map 
\[ \phi : (a_i)_{i \geq 0} \mapsto \lim_{n \to \infty} f_{a_0} \circ \cdots \circ f_{a_n} (x) \]
where $x$ is any point of $\Heis^d$. 

We assume that $\mathcal F$ satisfies the open set condition: there is an open subset $U \subset \Heis^d$ such that the $f_i(U)$ are 
pairwise disjoint and contained in $U$. Let $\tilde \mu$ be a product measure on $\Lambda^{\mathbf N}$. The image $\mu = \phi \tilde \mu$ is called  
a self-similar measure.

\begin{lemma}
There exist an ergodic fractal distribution $P$ such that any $P^\square$-typical measure $\nu$ is absolutely continuous with 
respect to $\tau \circ h \circ u (\mu)$ for some translation $\tau$, some homothety $h$, and some rotation $u \in H$.
\end{lemma}
\begin{proof}
The proof is identical to \cite{Hochman2010} paragraph 4.3. 
\end{proof}

The elements of $\mathcal F$ pass to the quotient and define a family of contractive similarity transformations of $\C^d$; let 
$\overline{\mathcal F}$ be this quotient family.
\begin{theorem}
    Assume that $\overline{\mathcal F}$ satisfies the open set condition. Then for any $x \in \Heis^d$, 
    \[ \dim (\pi_x \mu) = \dim (\pi_Z \mu) = \dim (\mu) \]
    and 
    \[ \dim (\pi_x X) = \dim (\pi_Z X) = \dim (X) \]
\end{theorem}
\begin{proof}
    Let $P$ be an ergodic fractal distribution as in the Lemma above. All we have to do is show that 
    \[ \dim(\pi_Z \mu) = \dim (\mu) \text. \]

    Let us disintegrate $\mu$ along $\pi_Z$:
    \[ \mu = \int \mathrm{d} (\pi_Z \mu) (u)\ \mu_u \]
    where $\mu_u$ is supported on $\pi_Z^{-1}(u)$ for $\pi_Z \mu$-almost every $u$. According to the Ledrappier-Young formula, 
    the conditional measure $\mu_u$ is almost surely exact dimensional, and its dimension is almost surely equal to a constant $\dim (\mu_u)$
    which satisfies the equality 
    \[ \dim(\mu) = \dim (\pi_Z \mu) + \dim (\mu_u) \]
    To apply the Ledrappier-Young formula, we do not need to assume that $\overline{\mathcal F}$ satisfies the open set condition; ergodicity 
    of $P$ is enough. The reason for our assumption that $\overline{\mathcal F}$ satisfies the open set condition is the following consequence: 
    in this case the restriction of $\pi_Z$ to the limit set 
    $X$ (which is equal to the support of $\mu$) is injective; hence $\mu_u$ is almost everywhere a Dirac mass, and has dimension $0$. In this 
    case the Ledrappier-Young formula becomes
     \[ \dim(\mu) = \dim (\pi_Z \mu) \]

     The first equality in the conclusion of the Theorem follows from this by virtue of Theorem \ref{th.proj-efd}; the second 
     equality is a consequence of the first because there is a self-similar measure $\mu$ such that $\dim(\mu) = \dim(X)$.

     We remark that the Ledrappier-Young formula has not been explicitly proved for ergodic fractal distributions in Heisenberg groups, only in 
     Euclidean spaces. See \cite{Dufloux2017} for a proof of the Ledrappier-Young formula that works in Heisenberg groups, in a slightly different setting 
     (conditional measures along group operations) that can be adapted to deal with ergodic fractal distributions.
\end{proof}

\bibliographystyle{plain}
\bibliography{bibli}

\end{document}